\documentclass[reqno]{amsart}
\usepackage{hyperref}

\begin{document}
\title[\hfilneg \hfil Uniqueness of Meromorphic Functions With Respect To Their Shifts Concerning  Derivatives]
{Uniqueness of Meromorphic Functions With Respect To Their Shifts Concerning  Derivatives}

\author[XiaoHuang Huang \hfil \hfilneg]
{XiaoHuang Huang}

\address{XiaoHuang Huang: Corresponding author\newline
Department of Mathematics, Shenzhen University, Shenzhen 518055, China}
\email{1838394005@qq.com}


\begin{abstract}
  An example in the article shows that the first derivative of $f(z)=\frac{2}{1-e^{-2z}}$ sharing $0$ CM and $1,\infty$ IM with its shift $\pi i$ cannot obtain they are equal. In this paper, we study the uniqueness of meromorphic function  sharing small functions with their shifts concerning its $k-th$ derivatives. We use a different method from Qi and Yang  \cite {qy} to improves entire function to meromorphic function, the first derivative to the $k-th$ derivatives, and also finite values to small functions. As for $k=0$, we obtain: Let $f(z)$ be a  transcendental meromorphic function of $\rho_{2}(f)<1$, let $c$ be a nonzero finite value, and let $a(z)\not\equiv\infty, b(z)\not\equiv\infty\in \hat{S}(f)$ be two distinct small functions of  $f(z)$ such that $a(z)$ is a periodic function with period $c$ and $b(z)$ is any small function of $f(z)$. If $f(z)$ and $f(z+c)$ share $a(z),\infty$ CM, and share $b(z)$ IM, then either $f(z)\equiv f(z+c)$ or $$e^{p(z)}\equiv \frac{f(z+c)-a(z+c)}{f(z)-a(z)}\equiv \frac{b(z+c)-a(z+c)}{b(z)-a(z)},$$
where $p(z)$ is a non-constant entire function of $\rho(p)<1$ such that $e^{p(z+c)}\equiv e^{p(z)}$.
\end{abstract}

\maketitle
\numberwithin{equation}{section}
\newtheorem{theorem}{Theorem}[section]
\newtheorem{lemma}[theorem]{Lemma}
\newtheorem{remark}[theorem]{Remark}
\newtheorem{corollary}[theorem]{Corollary}
\newtheorem{example}[theorem]{Example}
\newtheorem{problem}[theorem]{Problem}
\allowdisplaybreaks

\section{Introduction and main results}

Throughout this paper, we assume that the reader have a knowledge of  the fundamental results and the standard notations of the Nevanlinna value distribution theory. See(\cite{h3,y1,y2}). In the following, a meromorphic function $f$ means meromorphic in the whole complex plane. Define
 $$\rho(f)=\varliminf_{r\rightarrow\infty}\frac{log^{+}T(r,f)}{logr},$$
 $$\rho_{2}(f)=\varlimsup_{r\rightarrow\infty}\frac{log^{+}log^{+}T(r,f)}{logr}$$
by the order  and the hyper-order  of $f$, respectively. When $\rho(f)<\infty$, we say $f$ is of finite order.

By $S(r,f)$, we denote any quantity satisfying $S(r, f) = o(T(r, f))$, as $r\to \infty $ outside of a possible exceptional set of finite logarithmic measure. A meromorphic function $a(z)$ satisfying $T(r,a)=S(r,f)$ is called a small function of $f$.  We denote $S(f)$ as the family of all small meromorphic functions of $f$ which includes the constants in $\mathbb{C}$. Moreover, we define $\hat{S}(f)=S(f)\cup\{\infty\}$. We say that two non-constant meromorphic functions $f$ and $g$ share small function $a$ CM(IM) if $f-a$ and $g-a$ have the same zeros counting multiplicities (ignoring multiplicities). Moreover, we introduce the following notation: $S_{(m,n)}(a)=\{z|z $ is a common zero of $f(z+c)-a(z)$ and $f(z)-a(z)$ with multiplicities $m$ and $n$ respectively$\}$. $\overline{N}_{(m,n)}(r,\frac{1}{f-a})$ denotes the counting function of $f$ with respect to the set $S_{(m,n)}(a)$. $\overline{N}_{n)}(r,\frac{1}{f-a})$ denotes the counting function of all distinct zeros of $f-a$ with multiplicities at most $n$. $\overline{N}_{(n}(r,\frac{1}{f-a})$ denotes the counting function of all zeros of $f-a$ with multiplicities at least $n$.

We say that two non-constant meromorphic functions $f$ and $g$ share small function $a$ CM(IM)almost if
 $$N(r,\frac{1}{f-a})+N(r,\frac{1}{g-a})-2N(r,f=a=g)=S(r,f)+S(r,g),$$
or
$$\overline{N}(r,\frac{1}{f-a})+\overline{N}(r,\frac{1}{g-a})-2\overline{N}(r,f=a=g)=S(r,f)+S(r,g),$$
respectively.

For a meromorphic function $f(z)$, we denote its shift by $f_{c}(z)=f(z + c)$.

Rubel and Yang \cite{ruy}  studied the uniqueness of an entire function concerning its first order derivative, and proved the following result.

\

{\bf Theorem A} \ Let $f(z)$ be a non-constant entire function, and let $a, b$ be two finite distinct complex values. If $f(z)$ and $f'(z)$
 share $a, b$ CM, then $f(z)\equiv f'(z)$.

Zheng and Wang \cite{zw} improved Theorem A and proved

\

{\bf Theorem B} \ Let $f(z)$ be a non-constant entire function, and let $a(z)\not\equiv\infty, b(z)\not\equiv\infty$ be two distinct small functions of $f(z)$. If $f(z)$ and $f^{(k)}(z)$ share $a(z), b(z)$ CM, then $f(z)\equiv f^{(k)}(z)$.

Li  and Yang  \cite {ly3} improved Theorem B and proved

\

{\bf Theorem C} \ Let $f(z)$ be a non-constant entire function, and let $a(z)\not\equiv\infty, b(z)\not\equiv\infty$ be two  distinct small functions of $f(z)$. If $f(z)$ and $f^{(k)}(z)$
 share $a(z)$ CM, and share $b(z)$ IM. Then $f(z)\equiv f^{(k)}(z)$.

Recently, the value distribution of meromorphic functions concerning difference analogue has become a popular research, see [1-9,  12-14, 16-18].
 Heittokangas et al \cite{hkl} obtained a similar result analogue of Theorem A concerning shifts.

\

{\bf Theorem D}
 Let $f(z)$ be a non-constant entire function of finite order, let $c$ be a nonzero finite complex value, and let $a, b$ be two finite distinct complex values.
If $f(z)$ and $f(z+c)$ share $a, b$ CM, then $f(z)\equiv f(z+c).$

In \cite {qly}, Qi-Li-Yang investigated the value sharing problem with respect to $f'(z)$ and $f(z+c)$. They proved

\

{\bf Theorem E}
 Let $f(z)$ be a non-constant entire function of finite order, and let $a, c$ be two nonzero finite complex values.
If $f'(z)$ and $f(z+c)$ share $0, a$ CM, then $f'(z)\equiv f(z+c).$

Recently, Qi and Yang \cite {qy} improved Theorem E and proved

\

{\bf Theorem F}
 Let $f(z)$ be a non-constant entire function of finite order, and let $a, c$ be two nonzero finite complex value.
If $f'(z)$ and $f(z+c)$ share $0$ CM and $a$ IM, then $f'(z)\equiv f(z+c).$

Of above theorem, it's naturally to ask whether the condition $0,a$ can be replaced by two distinct small functions, and $f'$ can be replaced by $f^{(k)}$?

In this article, we give a positive answer. In fact, we prove the following more general result.

\

{\bf Theorem 1} Let $f(z)$ be a  transcendental meromorphic function of $\rho_{2}(f)<1$, let $c$ be a nonzero finite value, $k$ be a positive integer, and let $a(z)\not\equiv\infty, b(z)\not\equiv\infty\in \hat{S}(f)$ be two distinct small functions. If $f^{(k)}(z)$ and $f(z+c)$ share $a(z),\infty$ CM, and share $b(z)$ IM, then $f^{(k)}(z)\equiv f(z+c)$.

\

{\bf Example 1} \cite{hf}
Let $f(z)=\frac{2}{1-e^{-2z}}$, and let $c=\pi i$. Then $f'(z)$ and $f(z+c)$ share $0$ CM and share $1,\infty$ IM, but $f'(z)\not\equiv f(z+c)$.

This example shows that for meromorphic functions, the conclusion of Theorem 1 doesn't hold even when  sharing $\infty$ CM is replaced by sharing $\infty$ IM when $k=1$. We believe there are examples for any $k$, but we can not construct them.

As for $k=0$, Li and Yi \cite{ly} obtained

\

{\bf Theorem G} Let $f(z)$ be a  transcendental entire function of $\rho_{2}(f)<1$, let $c$ be a nonzero finite value,  and let $a(z)\not\equiv\infty, b(z)\not\equiv\infty\in \hat{S}(f)$ be two distinct small functions. If $f(z)$ and $f(z+c)$ share $a(z)$ CM, and share $b(z)$ IM, then $f(z)\equiv f(z+c)$.

\

{\bf Remark 1} Theorem G holds when $f(z)$ is a non-constant meromorphic function of $\rho_{2}(f)<1$ such that $N(r,f)=S(r,f)$.

Heittokangas, et. \cite{hkl1} proved.

\

{\bf Theorem H} Let $f(z)$ be a  non-constant meromorphic function of finite order, let $c$ be a nonzero finite value, and let $a(z)\not\equiv\infty$, $b(z)\not\equiv\infty$  and $d(z)\not\equiv\infty\in \hat{S}(f)$ be three distinct small functions  such that $a(z)$, $b(z)$ and $d(z)$ are periodic functions with period $c$. If $f(z)$ and $f(z+c)$ share $a(z),b(z)$ CM, and $d(z)$ IM, then $f(z)\equiv f(z+c)$.

We can ask a question that whether the small periodic function $d(z)$ of $f(z)$ can be replaced by any small function of $f(z)$?\\

In this paper, we obtain our second result.

\

{\bf Theorem 2} Let $f(z)$ be a  transcendental meromorphic function of $\rho_{2}(f)<1$, let $c$ be a nonzero finite value, and let $a(z)\not\equiv\infty, b(z)\not\equiv\infty\in \hat{S}(f)$ be two distinct small functions of  $f(z)$ such that $a(z)$ is a periodic function with period $c$ and $b(z)$ is a small function of $f(z)$. If $f(z)$ and $f(z+c)$ share $a(z),\infty$ CM, and share $b(z)$ IM, then either $f(z)\equiv f(z+c)$ or $$e^{p(z)}\equiv \frac{f(z+c)-a(z+c)}{f(z)-a(z)}\equiv \frac{b(z+c)-a(z+c)}{b(z)-a(z)},$$
where $p(z)$ is a non-constant entire function of $\rho(p)<1$ such that $e^{p(z+c)}\equiv e^{p(z)}$.

We can obtain the following corollary from the proof of Theorem 2.

\

{\bf Corollary 1}  Under the same condition as in Theorem 2, then $f(z)\equiv f(z+c)$ holds if one of conditions satisfies\\
(i) $b(z)$ is a periodic function with period $nc$ ;\\
(ii) $\rho(b(z))<\rho(e^{p(z)})$;\\
(iii) $\rho(b(z))<1$.

\

{\bf Example 2} \ Let $f(z)=\frac{e^{z}}{1-e^{-2z}}$, and let $c=\pi i$. Then $f(z+c)=\frac{-e^{z}}{1-e^{-2z}}$, and $f(z)$ and $f(z+c)$ share $0,\infty$ CM, but $f(z)\not\equiv f(z+c)$.

\

{\bf Example 3} \ Let $f(z)=e^{z}$, and let $c=\pi i$. Then $f(z+c)=-e^{z}$, and $f(z)$ and $f(z+c)$ share $0,\infty$ CM, $f(z)$ and $f(z+c)$ attain
different values everywhere in the complex plane, but $f(z)\not\equiv f(z+c)$.

Above two examples of  show that "2CM+1IM" is necessary.

\

{\bf Example 4} Let $f(z)=e^{e^{z}}$,  then $f(z+\pi i)=\frac{1}{e^{e^{z}}}$. It is easy to verify that $f(z)$ and $f(z+\pi i)$ share $0, 1, \infty$ CM, but $f(z)=\frac{1}{f(z+\pi i)}$. On the other hand, we  obtain $f(z)=f(z+2\pi i)$.

Example 4 tells us that if we drop the assumption  $\rho_{2}(f)<1$, we can get another relation.

By Theorem 1 and Theorem 2, we still believe the latter situation of Theorem 2 can be removed, that is to say,  only the case $f(z)\equiv f(z+c)$ occurs. So we raise a conjecture here.

\

{\bf Conjecture} Under the same condition as in Theorem 2, is  $f(z)\equiv f(z+c)$ ?

\section{Some Lemmas}
\begin{lemma}\label{21l}\cite{h3} Let $f$ be a non-constant meromorphic function of $\rho_{2}(f)<1$,  and let $c$ be a non-zero complex number. Then
$$m(r,\frac{f(z+c)}{f(z)})=S(r, f),$$
for all r outside of a possible exceptional set E with finite logarithmic measure.
\end{lemma}

\begin{lemma}\label{22l}\cite{hk3,y1,y2} Let $f_{1}$ and $f_{2}$ be two non-constant meromorphic functions in $|z|<\infty$, then
$$N(r,f_{1}f_{2})-N(r,\frac{1}{f_{1}f_{2}})=N(r,f_{1})+N(r,f_{2})-N(r,\frac{1}{f_{1}})-N(r,\frac{1}{f_{2}}),$$
where $0<r<\infty$.
\end{lemma}

\begin{lemma}\label{23l}\cite{h3}  Let $f$ be a non-constant meromorphic function of $\rho_{2}(f)<1$,  and let $c$ be a non-zero complex number. Then
$$T(r,f(z))=T(r,f(z+c))+S(r,f),$$
for all r outside of a possible exceptional set E with finite logarithmic measure.
\end{lemma}

\begin{lemma}\label{24l} Let $f$ be a transcendental meromorphic function of $\rho_{2}(f)<1$ such that $\overline{N}(r,f)=S(r,f)$,  let $c$ be a nonzero constant, $k$ be a positive integer, and let $a(z)$ be a small function of $f(z+c)$ and $f^{(k)}(z)$. If $f(z+c)$ and $f^{(k)}(z)$ share $a(z),\infty$ CM, and $N(r,\frac{1}{f^{(k)}(z+c)-a^{(k)}(z)})=S(r,f)$, then $T(r,e^{p})=S(r,f)$, where $p$ is an entire function of order less than $1$.
\end{lemma}
\begin{proof}
 Since $f$ is a transcendental meromorphic function of $\rho_{2}(f)<1$, $\overline{N}(r,f)=S(r,f)$, and $f_{c}$ and $f^{(k)}$ share $a$ and $\infty$ CM, then there is an entire function $p$ of order less than $1$ such that
\begin{align}
f_{c}-a=e^{p}(f^{(k)}-a^{(k)}_{-c})+e^{p}(a^{(k)}_{-c}-a).
\end{align}
Suppose on the contrary that $T(r,e^{p})\neq S(r,f)$.\\
Set $g=f^{(k)}_{c}-a^{(k)}$. Differentiating (2.1) $k$ times we have
\begin{align}
g=(e^{p})^{(k)}g_{-c}+k(e^{p})^{(k-1)}g_{-c}'+\cdots+k(e^{p})'g_{-c}^{(k-1)}+e^{p}g_{-c}^{(k)}+B^{(k)},
\end{align}
where $B=e^{p}(a^{(k)}_{-c}-a)$.\\
It is easy to see that $g\not\equiv0$. Then we rewrite (2.2) as
\begin{eqnarray}
1-\frac{B^{(k)}}{g}=De^{p},
\end{eqnarray}
where
\begin{align}
D&=e^{-p}[(e^{p})^{(k)}\frac{g_{-c}}{g}+k(e^{p})^{(k-1)}\frac{g_{-c}'}{g}+\cdots\notag\\
&+k(e^{p})'\frac{g_{-c}^{(k-1)}}{g}+(e^{p})\frac{g_{-c}^{(k)}}{g}].
\end{align}
Since  $f$ is a transcendental meromorphic function with $\rho_{2}(f)<1$ and $f^{(k)}$ and $f_{c}$ share  $ \infty$ CM, we can see from $\overline{N}(r,f)=S(r,f)$, Lemma 2.1 and Lemma 2.3 that
\begin{eqnarray*}
\begin{aligned}
(1+o(1))N(r,f)+S(r,f)=N(r,f_{c})=N(r,f^{(k)}),
\end{aligned}
\end{eqnarray*}
and on the other hand
\begin{eqnarray*}
\begin{aligned}
k\overline{N}(r,f_{c})+N(r,f_{c})=N(r,f^{(k)}_{c}),  \overline{N}(r,f_{c})=\overline{N}(r,f^{(k)})=\overline{N}(r,f),
\end{aligned}
\end{eqnarray*}
which follows from above equalities that $N(r,f^{(k)})=N(r,f^{(k)}_{c})+S(r,f)$, and thus we can know that $g$ and $g_{-c}$ share $\infty$ CM almost.  It is easy to see from the assumption $f_{c}$ and $f^{(k)}$ share $\infty$ CM that there exists no simple pole point of $f_{c}$.  Now we  estimate $N(r,\frac{g_{-c}^{(i)}}{g})$. Let $z_{0}$ be a pole of $f$ with multiplicity  $n$, than $z_{0}$ is a pole of $g$ with multiplicity  $n+2k$, and also $z_{0}$ is a pole of $g_{-c}^{(i)}$ with multiplicity  $n+k+i$. Then we can see that $z_{0}$ is a zero point of $\frac{g_{-c}^{(i)}}{g}$ with $k-i$.  Let $z_{1}$ be a pole of $f_{c}$ with multiplicity  $m$, then $z_{1}$ is a pole of $g$ with multiplicity  $m+k$, and also $z_{1}$ is a pole of $g_{-c}^{(i)}$ with multiplicity  $m+i$. Then we can see that $z_{1}$ is a zero point of $\frac{g_{-c}^{(i)}}{g}$ with $k-i$. Note that $N(r,\frac{1}{f^{(k)}_{c}-a^{(k)}})=N(r,\frac{1}{g})=S(r,f)$, then $N(r,\frac{g_{-c}^{(i)}}{g})=S(r,f)$, and hence
\begin{align}
T(r,D)&\leq\sum_{i=0}^{k}(T(r,\frac{(e^{p})^{(i)}}{e^{p}})+T(r,\frac{C_{k}^{i}g_{-c}^{(k-i)}}{g}))+S(r,f)\notag\\
&\leq\sum_{i=0}^{k}(S(r,e^{p})+m(r,\frac{g_{-c}^{(i)}}{g_{-c}})+N(r,\frac{g_{-c}^{(i)}}{g}))+S(r,f)\notag\\
&=S(r,e^{p})+S(r,f),
\end{align}
where $C_{k}^{i}$ is a combinatorial number. By (2.1) and Lemma 2.1,  we get
\begin{align}
T(r,e^{p})&\leq T(r,f_{c})+T(r,f^{(k)})+S(r,f)\notag\\
&\leq 2T(r,f)+S(r,f).
\end{align}
Then it follows from (2.5) that $T(r,D)=S(r,f)$. Next we discuss  two cases.

{\bf Case1.} \quad $e^{-p}-D\not\equiv0$. Rewrite (2.3) as
\begin{align}
ge^{p}(e^{-p}-D)=B^{(k)}.
\end{align}
We claim that $D\equiv0$. Otherwise, using the Lemma 2.8 to $e^{-p}$, we get
\begin{eqnarray*}
\begin{aligned}
&m(r,\frac{1}{e^{-p}-D})+N(r,\frac{1}{e^{-p}-D})=T(r,e^{-p})\notag\\
&\leq \overline{N}(r,e^{-p})+\overline{N}(r,\frac{1}{e^{-p}})+\overline{N}(r,\frac{1}{e^{-p}-D})\notag\\
&+S(r,e^{p})=\overline{N}(r,\frac{1}{e^{-p}-D})+S(r,f)\notag\\
&\leq T(r,e^{-p})+S(r,f),
\end{aligned}
\end{eqnarray*}
that is to say
\begin{eqnarray*}
\begin{aligned}
T(r,e^{p})=T(r,e^{-p})+O(1)=\overline{N}(r,\frac{1}{e^{-p}-D})+S(r,f)
\end{aligned}
\end{eqnarray*}
and
\begin{eqnarray*}
\begin{aligned}
N(r,\frac{1}{e^{-p}-D})=N_{1}(r,\frac{1}{e^{-p}-D})+S(r,f).
\end{aligned}
\end{eqnarray*}
It follows form above two equalities that
\begin{eqnarray*}
\begin{aligned}
T(r,e^{p})=N_{1}(r,\frac{1}{e^{-p}-D})+S(r,f).
\end{aligned}
\end{eqnarray*}
Because the numbers of zeros and poles of $B^{(k)}$ are $S(r,f)$, we can see from (2.7) and $\overline{N}(r,f)=S(r,f)$ that the multiplicities of poles of $g$  are almost $1$. And then
\begin{eqnarray*}
\begin{aligned}
&N(r,f)+k\overline{N}(r,f)=N(r,g)+S(r,f)=N(r,\frac{1}{e^{-p}-D})+S(r,f)\notag\\
&=N_{1}(r,f)+S(r,f)\leq \overline{N}(r,f)+S(r,f)=S(r,f).
\end{aligned}
\end{eqnarray*}
it follows from above that $\overline{N}(r,\frac{1}{e^{-p}-D})=S(r,f)$. Then by Lemma 2.8 in the following we can obtain
\begin{align}
T(r,e^{p})&=T(r,e^{-p})+O(1)\notag\\
&\leq \overline{N}(r,e^{-p})+\overline{N}(r,\frac{1}{e^{-p}})+\overline{N}(r,\frac{1}{e^{-p}-D})\notag\\
&+S(r,e^{p})=S(r,f),
\end{align}
which contradicts with present assumption. Thus $D\equiv0$. Then by (2.7) we get
\begin{align}
g=B^{(k)}.
\end{align}
Integrating (2.9), we get
\begin{align}
f_{c}=e^{p}(a^{(k)}_{-c}-a)+P+a,
\end{align}
where $P$ is a polynomial of degree at most $k-1$.  (2.10) implies
\begin{align}
T(r,f_{c})=T(r,e^{p})+S(r,f).
\end{align}
Substituting (2.9) and (2.10) into (2.1) we can obtain
\begin{align}
e^{p}(a^{(k)}_{-c}-a)+P=e^{p+p_{-c}}L_{-c},
\end{align}
where $L_{-c}$ is the differential polynomial in $$ p'_{-c},\ldots,p^{(k)}_{-c}, a_{-2c}-a_{-c},(a_{-2c}-a_{-c})',\ldots,(a_{-2c}-a_{-c})^{(k)},$$
and it is a small function of $f(z+c)$. On the one hand
\begin{align}
2T(r,e^{p})&=T(r,e^{2p})= m(r,e^{2p})\notag\\
&\leq m(r,e^{p+p_{-c}})+m(r,\frac{e^{p}}{e^{p_{-c}}})\notag\\
&\leq T(r,e^{p+p_{-c}})+S(r,f).
\end{align}
On the other hand, we can prove similarly that
\begin{align}
T(r,e^{p+p_{-c}})\leq 2T(r,e^{p})+S(r,f).
\end{align}
So
\begin{align}
T(r,e^{p+p_{-c}})= 2T(r,e^{p})+S(r,f).
\end{align}
By (2.11), (2.12) and (2.15) we can get $T(r,e^{p})=2T(r,e^{p})+S(r,f)$, which is $T(r,e^{p})=S(r,f)$, a contradiction.

{\bf Case2.} \quad $e^{-p}-D\equiv0$. Immediately, we get $T(r,e^{p})=S(r,f)$, but it's impossible.

Of above discussions, we conclude that $T(r,e^{p})=S(r,f)$.
\end{proof}
\begin{lemma}\label{25l} Let $f$ be a transcendental meromorphic function of $\rho_{2}(f)<1$ such that $\overline{N}(r,f)=S(r,f)$, let $k$ be a positive integer and $c\neq0$  a complex value, and let $a\not\equiv\infty$ and $b\not\equiv\infty$ be two distinct small functions of $f$. Suppose
\[L(f_{c})=\left|\begin{array}{rrrr}f_{c}-a& &a-b \\
f'_{c}-a'& &a'-b'\end{array}\right|\]
and
\[L(f^{(k)})=\left|\begin{array}{rrrr}f^{(k)}-a& &a-b \\
f^{(k+1)}-a'& &a'-b'\end{array}\right|,\]
and $f_{c}$ and $f^{(k)}$ share $a,\infty$ CM, and share $b$ IM,  then $L(f_{c})\not\equiv0$ and $L(f^{(k)})\not\equiv0$.
\end{lemma}
\begin{proof}
Suppose that $L(f_{c})\equiv0$, then we can get $\frac{f'_{c}-a'}{f_{c}-a}\equiv\frac{a'-b'}{a-b}$. Integrating both side of above we can obtain $f_{c}-a=C_{1}(a-b)$, where $C_{1}$ is a nonzero constant. So by Lemma 2.3, we have $T(r,f)=T(r,f_{c})+S(r,f)=T(r,C(a-b)+a)=S(r,f)$, a contradiction. Hence $L(f_{c})\not\equiv0$.

Since $f^{(k)}$ and $f_{c}$ share $a$ CM and $b$ IM, and $f$ is a transcendental meromorphic function of $\rho_{2}(f)<1$ such that $\overline{N}(r,f)=S(r,f)$, then by the  Lemma 2.8, we get
\begin{align}
T(r,f_{c})&\leq \overline{N}(r,\frac{1}{f_{c}-a})+\overline {N}(r,\frac{1}{f_{c}-b})+\overline{N}(r,f_{c})+S(r,f)\notag\\
&= \overline {N}(r,\frac{1}{f^{(k)}-a})+\overline {N}(r,\frac{1}{f^{(k)}-b})+S(r,f)\notag\\
&\leq 2T(r,f^{(k)})+S(r,f).
\end{align}
Hence $a$ and $b$ are small functions of $f^{(k)}$. If $L(f^{(k)})\equiv0$, then we can get $f^{(k)}-a=C_{2}(a-b)$, where $C_{2}$ is a nonzero constant. And we get $T(r,f^{(k)})=S(r,f^{(k)})$. Combing (2.16) we obtain $T(r,f)=T(r,f_{c})+S(r,f)=T(r,C(a-b)+a)=S(r,f)$, a contradiction.
\end{proof}

\begin{lemma}\label{26l}  Let $f$ be a transcendental meromorphic function, let $k_{j}(j=1,2,\ldots,q)$ be  distinct constants, and let $a\not\equiv\infty$ and $b\not\equiv\infty$ be two distinct small functions of $f$ . Again let $d_{j}=a-k_{j}(a-b)$ $(j=1,2,\ldots,q)$. Then
$$m(r,\frac{L(f_{c})}{f_{c}-a})=S(r,f), \quad m(r,\frac{L(f_{c})}{f_{c}-d_{j}})=S(r,f).$$
for $1\leq i\leq q$ and
$$m(r,\frac{L(f_{c})f_{c}}{(f_{c}-d_{1})(f_{c}-d_{2})\cdots(f_{c}-d_{m})})=S(r,f),$$
where $L(f_{c})$ is defined as in Lemma 2.5, and $2\leq m\leq q$.
\end{lemma}
\begin{proof}
Obviously, we have
$$m(r,\frac{L(f_{c})}{f_{c}-a})\leq m(r,\frac{(a'-b')(f_{c}-a)}{f_{c}-a})+m(r,\frac{(a-b)(f'_{c}-a')}{f_{c}-a})=S(r,f),$$
and
$$\frac{L(f_{c})f_{c}}{(f_{c}-d_{1})(f_{c}-d_{2})\cdots(f_{c}-d_{q})}=\sum_{i=1}^{q}\frac{C_{i}L(f_{c})}{f_{c}-d_{i}},$$
where $C_{i}=\frac{d_{j}}{\prod\limits_{j\neq i}(d_{i}-d_{j})}$ are small functions of $f$. By Lemma 2.1 and above, we have
\begin{align}
&m(r,\frac{L(f_{c})f_{c}}{(f_{c}-d_{1})(f_{c}-d_{2})\cdots(f_{c}-d_{q})})=m(r,\sum_{i=1}^{q}\frac{C_{i}L(f_{c})}{f_{c}-d_{i}})\notag\\
&\leq\sum_{i=1}^{q}m(r,\frac{L(f_{c})}{f_{c}-d_{i}})+S(r,f)=S(r,f).
\end{align}
\end{proof}
\begin{lemma}\label{27l} Let $f$ and $g$ be are two non-constant meromorphic functions such that $\overline{N}(r,f)=S(r,f)$, and let $a\not\equiv\infty$ and $b\not\equiv\infty$ be two distinct small functions of  $f$ and $g$. If
$$H=\frac{L(f)}{(f-a)(f-b)}-\frac{L(g)}{(g-a)(g-b)}\equiv0,$$
where
$$L(f)=(a'-b')(f-a)-(a-b)(f'-a')$$
and
$$L(g)=(a'-b')(g-a)-(a-b)(g'-a').$$
And  if $f$ and $g$ share $a,\infty$ CM, and share $b$ IM, then either $2T(r,f)=\overline{N}(r,\frac{1}{f-a})+\overline{N}(r,\frac{1}{f-b})+S(r,f)$, or $f=g$.
\end{lemma}
\begin{proof}
Integrating $H$ which leads to
$$\frac{g-b}{g-a}=C\frac{f-b}{f-a},$$
where $C$ is a nonzero constant.\\

If $C=1$, then $f=g$. If $C\neq1$, then from above, we have
$$\frac{a-b}{g-a}\equiv \frac{(C-1)f-Cb+a}{f-a},$$
and
$$T(r,f)=T(r,g)+S(r,f)+S(r,g).$$
It follows that $N(r,\frac{1}{f-\frac{Cb-a}{C-1}})=N(r,\frac{1}{a-b})=S(r,f)$. Then by Lemma 2.8 in the following,
\begin{eqnarray*}
\begin{aligned}
T(r,f)&\leq \overline{N}(r,f)+\overline{N}(r,\frac{1}{f-a})+\overline{N}(r,\frac{1}{f-\frac{Cb-a}{C-1}})+S(r,f)\\
&\leq \overline{N}(r,\frac{1}{f-a})+S(r,f)\leq T(r,f)+S(r,f),
\end{aligned}
\end{eqnarray*}
and
\begin{eqnarray*}
\begin{aligned}
T(r,f)&\leq \overline{N}(r,f)+\overline{N}(r,\frac{1}{f-b})+\overline{N}(r,\frac{1}{f-\frac{Cb-a}{C-1}})+S(r,f)\\
&\leq \overline{N}(r,\frac{1}{f-b})+S(r,f)\leq T(r,f)+S(r,f),
\end{aligned}
\end{eqnarray*}
that is $T(r,f)=\overline{N}(r,\frac{1}{f-a})+S(r,f)$ and $T(r,f)=\overline{N}(r,\frac{1}{f-b})+S(r,f)$, and hence $2T(r,f)=\overline{N}(r,\frac{1}{f-a})+\overline{N}(r,\frac{1}{f-b})+S(r,f)$.
\end{proof}

\begin{lemma}\label{28l}\cite{y3} Let $f(z)$  be a non-constant meromorphic function, and let $a_{j}\in \hat{S}(f)$ be $q$ distinct small functions for all $j=1,2,\ldots,q$. Then
$$(q-2-\epsilon)T(r,f)\leq \sum_{j=1}^{q}\overline{N}(r,\frac{1}{f-a_{j}})+S(r,f), r\not\in E,$$
 for all r outside of a possible exceptional set E with finite logarithmic measure.
\end{lemma}

\

{\bf Remark 2} Lemma 2.8 is true when $\infty, a_{1}, a_{2}, \cdots,a_{q}\in \hat{S}(f)$ with $S(r,f)$ in our notation, in other words, even if exceptional sets are of infinite linear measure. but they are not of infinite logarithmic measure.

\begin{lemma}\label{2011} \cite{lh} Let $f$ and $g$ be two non-constant meromorphic functions. If $f$ and $g$ share $0,1,\infty$ IM, and $f$ is  a bilinear transformation of $g$,  then $f$ and $g$ assume one of the following six relations: (i) $fg=1$; (ii) $(f-1)(g-1)=1$; (iii) $f+g=1$; (iv) $f=cg$; (v) $f-1=c(g-1)$; (vi) $[(c-1)f+1][(c-1)g-c]=-c$, where $c\neq0,1$ is a complex number.
\end{lemma}

\begin{lemma}\label{2010}\cite{g}
 Let $f$, $F$ and $g$ be three non-constant meromorphic functions, where $g=F(f)$. Then $f$ and $g$ share three values IM if and only if there exist an entire function $h$ such that,
 by a  suitable linear fractional transformation, one of the following cases holds: \\
 (i) $f\equiv g$;\\
 (ii) $f=e^{h}$ and $g=a(1+4ae^{-h}-4a^{2}e^{-2h})$ have three IM shared values $a\neq0$, $b=2a$ and $\infty$;\\
 (iii) $f=e^{h}$ and $g=\frac{1}{2}(e^{h}+a^{2}e^{-h})$ have three IM shared values $a\neq0$, $b=-a$ and $\infty$;\\
 (iv) $f=e^{h}$ and $g=a+b-abe^{-h}$ have three IM shared values $ab\neq0$ and $\infty$;\\
 (v) $f=e^{h}$ and $g=\frac{1}{b}e^{2h}-2e^{h}+2b$ have three IM shared values $b\neq0$, $a=2b$ and $\infty$;\\
 (vi) $f=e^{h}$ and $g=b^{2}e^{-h}$ have three IM shared values $a\neq0$, $0$ and $\infty$.
 \end{lemma}

\begin{lemma}\label{2011} \cite{hk3,y1,y2}  Let $f$ and $g$ be two non-constant meromorphic functions, and let $\rho(f)$ and  $\rho(g)$ be the order of $f$ and $g$, respectively. Then
$\rho(fg)\leq \max\{\rho(f), \rho(g)\}$.
\end{lemma}

\

{\bf Remark 3} We can see from the proof that Lemma 2.9 \cite{lh} and Lemma 2,10 \cite{y1} are still true when $f$ and $g$ share three value IM almost.

\section{The proof of Theorem 1}
If $f_{c}\equiv f^{(k)}$, there is nothing to prove. Suppose $f_{c}\not\equiv f^{(k)}$. Since $f$ is a non-constant meromorphic function of $\rho_{2}(f)<1$,  $f_{c}$ and $f^{(k)}$ share $a,\infty$ CM, then  we get
\begin{align}
\frac{f^{(k)}-a}{f_{c}-a}=e^{h},
\end{align}
where $h$ is an entire function, and it is easy to know from (2.1) that $h=-p$.

Since  $f$ is a transcendental meromorphic function of $\rho_{2}(f)<1$ and $f^{(k)}$ and $f_{c}$ share  $ \infty$ CM, we can see from Lemma 2.1 and Lemma 2.3 that
\begin{eqnarray*}
\begin{aligned}
(1+o(1))N(r,f)+S(r,f)=N(r,f_{c})=N(r,f^{(k)}),
\end{aligned}
\end{eqnarray*}
which implies
\begin{eqnarray*}
\begin{aligned}
\overline{N}(r,f)=S(r,f).
\end{aligned}
\end{eqnarray*}
Furthermore, from the assumption that $f^{(k)}$ and $f_{c}$ share $a$ and $ \infty$ CM and $b$ IM, then by  Lemma 2.1, Lemma 2.8 and above equality, we get
\begin{eqnarray*}
\begin{aligned}
T(r,f_{c})&\leq \overline{N}(r,\frac{1}{f_{c}-a})+\overline {N}(r,\frac{1}{f_{c}-b})+\overline{N}(r,f_{c})+S(r,f)\\
&= \overline {N}(r,\frac{1}{f^{(k)}-a})+\overline {N}(r,\frac{1}{f^{(k)}-b})+S(r,f)\\
&\leq N(r,\frac{1}{f_{c}-f^{(k)}})+S(r,f)\leq T(r,f_{c}-f^{(k)})+S(r,f)\\
&\leq m(r,f_{c}-f^{(k)})+N(r,f_{c}-f^{(k)})+S(r,f)\\
&\leq m(r,f_{c})+m(r,1-\frac{f^{(k)}}{f_{c}})+N(r,f_{c})+S(r,f)\leq T(r,f_{c})+S(r,f).
\end{aligned}
\end{eqnarray*}

That is
\begin{eqnarray}
T(r,f_{c})=\overline {N}(r,\frac{1}{f_{c}-a})+\overline {N}(r,\frac{1}{f_{c}-b})+S(r,f).
\end{eqnarray}

By (3.1) and (3.2) we have
\begin{eqnarray}
T(r,f_{c})=T(r,f_{c}-f^{(k)})+S(r,f)=N(r,\frac{1}{f_{c}-f^{(k)}})+S(r,f).
\end{eqnarray}
and by Lemma 2.1,
\begin{align}
&T(r,e^{h})=m(r,e^{h})=m(r,\frac{f^{(k)}-a_{-c}^{(k)}+a_{-c}^{(k)}-a}{f_{c}-a})\leq m(r,\frac{a_{-c}^{(k)}-a}{f_{c}-a})\notag\\
&+ m(r,\frac{f^{(k)}-a_{-c}^{(k)}}{f^{(k)}_{c}-a^{(k)}})+m(r,\frac{f^{(k)}_{c}-a^{(k)}}{f_{c}-a})\leq m(r,\frac{1}{f_{c}-a})+S(r,f).
\end{align}
Then it follows from (3.1) and (3.3) that
\begin{align}
&m(r,\frac{1}{f_{c}-a})=m(r,\frac{e^{h}-1}{f^{(k)}-f_{c}})\notag\\
&\leq m(r,\frac{1}{f^{(k)}-f_{c}})+m(r,e^{h}-1)\notag\\
&\leq T(r,e^{h})+S(r,f).
\end{align}
Then by (3.4) and (3.5)
\begin{align}
T(r,e^{h})= m(r,\frac{1}{f_{c}-a})+S(r,f).
\end{align}
On the other hand, (3.1) can be rewritten as
\begin{align}
\frac{f^{(k)}-f_{c}}{f_{c}-a}=e^{h}-1,
\end{align}
which implies
\begin{align}
\overline{N}(r,\frac{1}{f_{c}-b})\leq \overline{N}(r,\frac{1}{e^{h}-1})+S(r,f)=T(r,e^{h})+S(r,f).
\end{align}
Thus, by (3.2), (3.6) and (3.8)
\begin{eqnarray*}
\begin{aligned}
m(r,\frac{1}{f_{c}-a})+N(r,\frac{1}{f_{c}-a})&= \overline{N}(r,\frac{1}{f_{c}-a})+\overline{N}(r,\frac{1}{f_{c}-b})+S(r,f)\\
&\leq \overline{N}(r,\frac{1}{f_{c}-a})+\overline{N}(r,\frac{1}{e^{h}-1})+S(r,f)\\
&\leq\overline{N}(r,\frac{1}{f_{c}-a})+m(r,\frac{1}{f_{c}-a})+S(r,f),
\end{aligned}
\end{eqnarray*}
which implies
\begin{align}
N(r,\frac{1}{f_{c}-a})=\overline{N}(r,\frac{1}{f_{c}-a})+S(r,f).
\end{align}
And then
\begin{align}
\overline{N}(r,\frac{1}{f_{c}-b})=T(r,e^{h})+S(r,f).
\end{align}
Set
\begin{eqnarray}
\varphi=\frac{L(f_{c})(f_{c}-f^{(k)})}{(f_{c}-a)(f_{c}-b)},
\end{eqnarray}
and
\begin{eqnarray}
\psi=\frac{L(f^{(k)})(f_{c}-f^{(k)})}{(f^{(k)}-a)(f^{(k)}-b)}.
\end{eqnarray}
 It is easy to know that $\varphi\not\equiv0$ because of Lemma 2.5 and $f\not\equiv f^{(k)} $. We know that $N(r,\varphi)\leq \overline{N}(r,f)=S(r,f)$ by (3.11). By Lemma 2.1 and Lemma 2.6 we have
\begin{eqnarray*}
\begin{aligned}
T(r,\varphi)&=m(r,\varphi)+N(r,\varphi)=m(r,\frac{L(f_{c})(f_{c}-f^{(k)})}{(f_{c}-a)(f_{c}-b)})+S(r,f)\notag\\
&\leq m(r,\frac{L(f_{c})f_{c}}{(f_{c}-a)(f_{c}-b)})+m(r,1-\frac{f^{(k)}}{f_{c}})+S(r,f)=S(r,f),
\end{aligned}
\end{eqnarray*}
that is
\begin{align}
T(r,\varphi)=S(r,f).
\end{align}
Let $d=a-j(a-b)(j\neq0,1)$. Obviously, by Lemma 2.1 and Lemma 2.6, we obtain
\begin{align}
 m(r,\frac{1}{f_{c}})&=m(r,\frac{1}{(b-a)\varphi}(\frac{L(f_{c})}{f_{c}-a}-\frac{L(f_{c})}{f_{c}-b})(1-\frac{f^{(k)}}{f_{c}}))\notag\\
&\leq m(r,\frac{1}{\varphi})+m(r,\frac{L(f_{c})}{f_{c}-a}-\frac{L(f_{c})}{f_{c}-b})\notag\\
&+m(r,1-\frac{f^{(k)}}{f_{c}})+S(r,f)=S(r,f).
\end{align}
and
\begin{align}
m(r,\frac{1}{f_{c}-d})&=m(r,\frac{L(f_{c})(f_{c}-f^{(k)})}{\varphi (f_{c}-a)(f_{c}-b)(f_{c}-d)})\notag\\
&\leq m(r,1-\frac{f^{(k)}}{f_{c}})+m(r,\frac{L(f_{c})f_{c}}{(f_{c}-a)(f_{c}-b)(f_{c}-d)})\notag\\
&+S(r,f)=S(r,f).
\end{align}
Set
\begin{align}
\phi=\frac{L(f_{c})}{(f_{c}-a)(f_{c}-b)}-\frac{L(f^{(k)})}{(f^{(k)}-a)(f^{(k)}-b)}.
\end{align}
We  discuss  two cases.\\

Case 1\quad $\phi\equiv0$.  Integrating the both sides of (3.16) which leads to
\begin{align}
\frac{f_{c}-a}{f_{c}-b}=C\frac{f^{(k)}-a}{f^{(k)}-b},
\end{align}
where $C$ is a nonzero constant.
Then by Lemma 2.7 we get
\begin{eqnarray}
2T(r,f_{c})=\overline{N}(r,\frac{1}{f_{c}-a})+\overline{N}(r,\frac{1}{f_{c}-b})+S(r,f),
\end{eqnarray}
which contradicts with (3.2).

Case 2 \quad $\phi \not\equiv0$. By (3.3), (3.13) and (3.16) we can obtain
\begin{align}
T(r,f_{c})&=T(r,f_{c}-f^{(k)})+S(r,f)=T(r,\frac{\phi(f_{c}-f^{(k)})}{\phi})+S(r,f)\notag\\
&=T(r,\frac{\varphi-\psi}{\phi})+S(r,f)\leq T(r,\varphi-\psi)+T(r,\phi)+S(r,f)\notag\\
&\leq T(r,\psi)+T(r,\phi)+S(r,f)\leq T(r,\psi)+\overline{N}(r,\frac{1}{f_{c}-b})+S(r,f).
\end{align}
On the other hand,
\begin{align}
T(r,\psi)&=T(r,\frac{L(f^{(k)})(f_{c}-f^{(k)})}{(f^{(k)}-a)(f^{(k)}-b)})\notag\\
&=m(r,\frac{L(f^{(k)})(f_{c}-f^{(k)})}{(f^{(k)}-a)(f^{(k)}-b)})+N(r,\psi)\notag\\
&\leq m(r,\frac{L(f^{(k)})}{f^{(k)}-b})+m(r,\frac{f_{c}-f^{(k)}}{f^{(k)}-a})+\overline{N}(r,f)+S(r,f)\notag\\
&\leq m(r,\frac{1}{f_{c}-a})+S(r,f)=\overline{N}(r,\frac{1}{f_{c}-b})+S(r,f).
\end{align}
Hence combining  (3.19) and (3.20), we obtain
\begin{align}
 T(r,f_{c})\leq 2\overline{N}(r,\frac{1}{f_{c}-b})+S(r,f).
\end{align}
If $a^{(k)}_{-c}\equiv a$, then by (3.1) and Lemma 2.1 we can get
\begin{align}
&T(r,e^{h})= m(r,e^{h})=m(r,\frac{f^{(k)}-a^{(k)}_{-c}}{f_{c}-a})\notag\\
&\leq m(r,\frac{f^{(k)}-a^{(k)}_{-c}}{f_{c}^{(k)}-a^{(k)}})+m(r,\frac{f_{c}^{(k)}-a^{(k)}}{f_{c}-a})=S(r,f).
\end{align}
It follows from (3.10), (3.21), (3.22) and Lemma 2.3 that $T(r,f)=T(r,f_{c})+S(r,f)=S(r,f)$. It's impossible.

If $a^{(k)}_{-c}\equiv b$, then by (3.10), (3.21) and and Lemma 2.1,
\begin{eqnarray*}
\begin{aligned}
T(r,f_{c})&\leq m(r,\frac{1}{f_{c}-a})+\overline{N}(r,\frac{1}{f^{(k)}-b})+S(r,f)\\
&\leq m(r,\frac{f^{(k)}-a_{-c}^{(k)}}{f^{(k)}_{c}-a^{(k)}})+m(r,\frac{f^{(k)}_{c}-a^{(k)}}{f_{c}-a})+m(r,\frac{1}{f^{(k)}-b})\\
&+\overline{N}(r,\frac{1}{f^{(k)}-b})+S(r,f)\leq T(r,f^{(k)})+S(r,f),
\end{aligned}
\end{eqnarray*}
which implies
\begin{align}
T(r,f_{c})\leq T(r,f^{(k)})+S(r,f).
\end{align}
Lemma 2.3 implies
\begin{align}
 T(r,f^{(k)})\leq T(r,f)+k\overline{N}(r,f)+S(r,f)=T(r,f_{c})+S(r,f),
\end{align}
and it follows from the fact $f_{c}$ and $f^{(k)}$ share $a$ CM and $b$ IM, (3.2) and (3.23) that
\begin{align}
 T(r,f^{(k)})&=T(r,f_{c})+S(r,f)\notag\\
& =\overline{N}(r,\frac{1}{f_{c}-a})+\overline{N}(r,\frac{1}{f_{c}-b})+S(r,f)\notag\\
 & =\overline{N}(r,\frac{1}{f^{(k)}-a})+\overline{N}(r,\frac{1}{f^{(k)}-b})+S(r,f).
  \end{align}
By  Lemma 2.1, Lemma 2.8, (3.2) and (3.25), we have
\begin{eqnarray*}
\begin{aligned}
2T(r,f^{(k)})&\leq\overline{N}(r,\frac{1}{f^{(k)}-a})+\overline{N}(r,\frac{1}{f^{(k)}-b})+\overline{N}(r,\frac{1}{f^{(k)}-d})+\overline{N}(r,f^{(k)})\\
&+S(r,f)\leq 2T(r,f^{(k)})-m(r,\frac{1}{f^{(k)}-d})+S(r,f)
\end{aligned}
\end{eqnarray*}
Immediately,
\begin{eqnarray}
m(r,\frac{1}{f^{(k)}-d})=S(r,f).
\end{eqnarray}

By the First Fundamental Theorem, Lemma 2.1, Lemma 2.2,  (3.14), (3.25), (3.26) and  $f$ is a transcendental meromorphic function of $\rho_{2}(f)<1$, we obtain
\begin{eqnarray*}
\begin{aligned}
m(r,\frac{f_{c}-d}{f^{(k)}-d})&\leq m(r,\frac{f_{c}}{f^{(k)}-d})+m(r,\frac{d}{f^{(k)}-d})+O(1)\\
&\leq T(r,\frac{f_{c}}{f^{(k)}-d})-N(r,\frac{f_{c}}{f^{(k)}-d})+S(r,f)\\
&=m(r,\frac{f^{(k)}-d}{f_{c}})+N(r,\frac{f^{(k)}-d}{f_{c}})-N(r,\frac{f_{c}}{f^{(k)}-d})+S(r,f)\\
&\leq N(r,\frac{1}{f_{c}})-N(r,\frac{1}{f^{(k)}-d})+N(r,f^{(k)})-N(r,f)+S(r,f)\\
&=T(r,\frac{1}{f_{c}})-T(r,\frac{1}{f^{(k)}-d})+S(r,f)\\
&=T(r,f_{c})-T(r,f^{(k)})+S(r,f)=S(r,f).
\end{aligned}
\end{eqnarray*}

Thus
\begin{eqnarray}
m(r,\frac{f_{c}-d}{f^{(k)}-d})=S(r,f).
\end{eqnarray}
It's easy to see that $N(r,\psi)=S(r,f)$ and (3.12) can be rewritten as
\begin{eqnarray}
\psi=[\frac{a-d}{a-b}\frac{L(f^{(k)})}{f^{(k)}-a}-\frac{b-d}{a-b}\frac{L(f^{(k)})}{f^{(k)}-b}][\frac{f_{c}-d}{f^{(k)}-d}-1].
\end{eqnarray}
Then by Lemma 2.6, (3.27) and (3.28) we can get
\begin{eqnarray}
T(r,\psi)=m(r,\psi)+N(r,\psi)=S(r,f).
\end{eqnarray}
By (3.2), (3.19) and (3.29) we get
\begin{eqnarray}
\overline{N}(r,\frac{1}{f_{c}-a})=S(r,f).
\end{eqnarray}
Moreover, by Lemma 2.1, (3.2), (3.25) and (3.30), we have
\begin{eqnarray}
m(r,\frac{1}{(f_{c}-a)^{(k)}})=m(r,\frac{1}{f^{(k)}_{c}-b_{c}})=m(r,\frac{1}{f^{(k)}-b})+S(r,f)=S(r,f),
\end{eqnarray}
and it follows from above, (3.6) and  (3.10) that
\begin{align}
 &\overline{N}(r,\frac{1}{f_{c}-b})=m(r,\frac{1}{f_{c}-a})+S(r,f)\notag\\
 &\leq m(r,\frac{1}{(f_{c}-a)^{(k)}})+m(r,\frac{(f_{c}-a)^{(k)}}{f_{c}-a})+S(r,f)=S(r,f).
\end{align}
Then by (3.2), (3.30), (3.32) and Lemma 2.3, we obtain
\begin{align}
T(r,f)&=T(r,f_{c})+S(r,f)=\overline{N}(r,\frac{1}{f_{c}-a})\notag\\
&+\overline{N}(r,\frac{1}{f_{c}-b})+S(r,f)=S(r,f),
\end{align}
which implies $T(r,f)=S(r,f)$, a contradiction.\\

So by (3.6), (3.10), (3.21), the First Fundamental Theorem, and Lemma 2.8 we can get
\begin{eqnarray*}
\begin{aligned}
T(r,f_{c})&\leq 2m(r,\frac{1}{f_{c}-a})+S(r,f)\leq2m(r,\frac{1}{f^{(k)}-a_{-c}^{(k)}})\\
&+S(r,f)=2T(r,f^{(k)})-2N(r,\frac{1}{f^{(k)}-a_{-c}^{(k)}})+S(r,f)\\
&\leq\overline{N}(r,\frac{1}{f^{(k)}-a})+\overline{N}(r,\frac{1}{f^{(k)}-b})+\overline{N}(r,\frac{1}{f^{(k)}-a_{-c}^{(k)}})\\
&+\overline{N}(r,f^{(k)})-2N(r,\frac{1}{f^{(k)}-a_{-c}^{(k)}})+S(r,f)\\
&\leq T(r,f_{c})-N(r,\frac{1}{f^{(k)}-a_{-c}^{(k)}})+S(r,f),
\end{aligned}
\end{eqnarray*}
which implies that
\begin{align}
N(r,\frac{1}{f^{(k)}-a_{-c}^{(k)}})=S(r,f).
\end{align}
Consequently, Lemma 2.1 and Lemma 2.3 can deduce
$$N(r,\frac{1}{f^{(k)}-a_{-c}^{(k)}})=N(r,\frac{1}{f_{c}^{(k)}-a^{(k)}})=S(r,f).$$
Then applying Lemma 2.4, we have  $T(r,e^{h})=T(r,e^{p})+O(1)=S(r,f)$, and it follows from (3.10) and (3.21) we can get $T(r,f)=T(r,f_{c})+S(r,f)=S(r,f)$, a contradiction.
This completes the proof of Theorem 1.

\section{The Proof of Theorem 2}
 If $f(z)\equiv f(z+c)$, there is nothing to do. Assume that $f(z)\not\equiv f(z+c)$. Since $f(z)$ is a transcendental meromorphic function of $\rho_{2}(f)<1$, $f$ and $f(z+c)$ share $a(z),\infty$ CM, then there is a nonzero entire function $p(z)$  of order less than $1$ such that
\begin{eqnarray}
\frac{f(z+c)-a(z)}{f(z)-a(z)}=e^{p(z)},
\end{eqnarray}
then by Lemma 2.1 and $a(z)$ is a periodic function with period $c$,
\begin{eqnarray}
T(r,e^{p})=m(r,e^{p})=m(r,\frac{f(z+c)-a(z+c)}{f(z)-a(z)})=S(r,f).
\end{eqnarray}
On the other hand, (4.1) can be rewritten as
\begin{eqnarray}
\frac{f(z+c)-f(z)}{f(z)-a(z)}=e^{p(z)}-1,
\end{eqnarray}
and then we get
\begin{eqnarray}
\overline{N}(r,\frac{1}{f(z)-b(z)})\leq N(r,\frac{1}{e^{p(z)}-1})=S(r,f).
\end{eqnarray}

Denote $N_{(m,n)}(r,\frac{1}{f(z)-b(z)})$ by the  zeros of $f(z)-b(z)$ with multiplicities $m$ and the zeros of $f_{c}(z)-b(z)$ with multiplicities $n$, where $m,n$ are two positive integers.  Thus, we can obtain
\begin{align}
&N(r,\frac{1}{f(z)-b(z)})=\sum_{k=2}^{n}N_{(1,k)}(r,\frac{1}{f(z)-b(z)})+\sum_{l=2}^{m}N_{(l,1)}(r,\frac{1}{f(z)-b(z)})\notag\\
&+\sum_{l=2}^{m}\sum_{k=2}^{n}N_{(l,k)}(r,\frac{1}{f(z)-b(z)})\leq \overline{N}(r,\frac{1}{f(z)-b(z)})+m\overline{N}(r,\frac{1}{f(z+c)-b(z)})\notag\\
&+N(r,\frac{1}{e^{p(z)}-1})\leq (m+1)\overline{N}(r,\frac{1}{f(z)-b(z)})+S(r,f)=S(r,f),
\end{align}
that is
\begin{eqnarray}
N(r,\frac{1}{f(z+c)-b(z+c)})=N(r,\frac{1}{f(z)-b(z)})=S(r,f).
\end{eqnarray}
Similarly, we also have
\begin{eqnarray}
N(r,\frac{1}{f(z+c)-b(z)})=S(r,f).
\end{eqnarray}

Set
\begin{eqnarray}
\psi(z)=\frac{f(z+c)-b(z+c)}{f(z)-b(z)}.
\end{eqnarray}
It is easy to see that
\begin{eqnarray}
N(r,\frac{1}{\psi(z)})\leq N(r,\frac{1}{f(z+c)-b(z+c)})+N(r,b(z))= S(r,f),
\end{eqnarray}
\begin{eqnarray}
N(r,\psi(z))\leq N(r,\frac{1}{f(z)-b(z)})+N(r,b(z))= S(r,f).
\end{eqnarray}
Hence by Lemma 2.1 and above,
\begin{align}
T(r,\psi(z))&=m(r,\psi(z))+N(r,\psi(z))=S(r,f)
\end{align}
According to (4.1) and (4.8),we have
\begin{eqnarray}
(e^{p(z)}-\psi(z))f(z)+\psi(z)b(z)+a(z)-b(z+c)-a(z)e^{p(z)}\equiv0.
\end{eqnarray}
We discuss following two cases.\\

{\bf Case 1} \quad $e^{p(z)}\not\equiv\psi(z)$. Then by (4.2), (4.11) and (4.12) we obtain $T(r,f)=S(r,f)$, a contradiction.\\

{\bf Case 2} \quad $e^{p(z)}\equiv\psi(z)$. Then by (4.1) we have
\begin{eqnarray}
f(z+c)=e^{p(z)}(f(z)-a(z))+a(z),
\end{eqnarray}
and
\begin{eqnarray}
N(r,\frac{1}{f(z+c)-b(z)})=N(r,\frac{1}{f(z)-a(z)+\frac{a(z)-b(z)}{e^{p(z)}}})=S(r,f).
\end{eqnarray}
If $b(z)$ is a periodic function of period $c$, then by (4.12) we can get $e^{p(z)}\equiv1$, which implies $f(z)\equiv f(z+c)$, a contradiction. Obviously, $a(z)-\frac{a(z)-b(z)}{e^{p(z)}}\not\equiv a(z)$. Otherwise, we can deduce $a(z)\equiv b(z)$, a contradiction.\\

Next, we discuss three Subcases.

{\bf Subcase 2.1}\quad $a(z)-\frac{a(z)-b(z)}{e^{p(z)}}\not\equiv b(z)$ and $a(z)-\frac{a(z)-b(z)}{e^{p(z)}}\not\equiv b(z-c)$. Then according to (4.6), (4.7),(4.14) and Lemma 2.8, we can get
\begin{align}
T(r,f(z))&\leq \overline{N}(r,\frac{1}{f(z)-a(z)-\frac{a(z)-b(z)}{e^{p(z)}}})+\overline{N}(r,\frac{1}{f(z)-b(z)})\notag\\
&+\overline{N}(r,\frac{1}{f(z)-b(z-c)})+S(r,f)=S(r,f),
\end{align}
that is $T(r,f(z))=S(r,f)$, a contradiction.\\

{\bf Subcase 2.2}\quad $a(z)-\frac{a(z)-b(z)}{e^{p(z)}}\equiv b(z)$, but $a(z)-\frac{a(z)-b(z)}{e^{p(z)}}\not\equiv b(z-c)$. It follows that $e^{p(z)}\equiv1$. Therefore by (4.1) we have $f(z)\equiv f(z+c)$, a contradiction.

{\bf Subcase 2.3}\quad $a(z)-\frac{a(z)-b(z)}{e^{p(z)}}\equiv b(z)$,  $a(z)-\frac{a(z)-b(z)}{e^{p(z)}}\equiv b(z-c)$. It follows that $e^{p(z)}\equiv1$. Therefore by (4.1) we have $f(z)\equiv f(z+c)$, a contradiction.

{\bf Subcase 2.4}\quad $a(z)-\frac{a(z)-b(z)}{e^{p(z)}}\not\equiv b(z)$ and $a(z)-\frac{a(z)-b(z)}{e^{p(z)}}\equiv b(z-c)$. It is easy to see that
\begin{eqnarray}
\frac{a(z)-b(z)}{a(z-c)-b(z-c)}=e^{p(z)}.
\end{eqnarray}
Furthermore, (4.12) implies
\begin{eqnarray}
\frac{a(z+c)-b(z+c)}{a(z)-b(z)}=e^{p(z)},
\end{eqnarray}
\begin{eqnarray}
\frac{a(z)-b(z)}{a(z-c)-b(z-c)}=e^{p(z-c)}.
\end{eqnarray}
It follows from (4.16) and (4.18) that
\begin{eqnarray}
e^{p(z)}=e^{p(z+c)}.
\end{eqnarray}
By (4.1), (4.8) and (4.19), we know that $f(z)$ and $f(z+nc)$ share $a(z)$ and $\infty$ CM, so we set
\begin{align}
F(z)=\frac{f(z)-a(z)}{b(z)-a(z)}, \quad G(z)=\frac{f(z+nc)-a(z)}{b(z+nc)-a(z+nc)}.
\end{align}
Since $f(z)$ and $f(z+nc)$ share $a(z)$ and $\infty$ CM, and $(b(z),b(z+nc)$ CM, so $F(z)$ and $G(z)$ share $0,\infty$ CM almost, and $1$ CM almost. We claim that $F$ is not a bilinear transform of $G$. Otherwise, we can see from Lemma 2.9 that if (i) occurs, we have $N(r,f(z))=N(r,F(z))+S(r,f)=S(r,f)$, then by Remark 1 and Theorem G, we get $f(z)\equiv f(z+c)$, a contradiction.

If (ii) occurs, we have $N(r,f(z))=N(r,F(z))+S(r,f)=S(r,f)$, then by Remark 1 and Theorem G, we get $f(z)\equiv f(z+c)$, a contradiction.

If (iii) occurs, we have
\begin{align}
N(r,\frac{1}{f(z)-a(z)})=S(r,f ), \quad N(r,\frac{1}{f(z)-b(z)})=S(r,f).
\end{align}
Then it follows from above, $a(z)-\frac{a(z)-b(z)}{e^{p(z)}}\not\equiv a(z)$, $a(z)-\frac{a(z)-b(z)}{e^{p(z)}}\not\equiv b(z)$ and Lemma 2.8 that  $T(r,f)=S(r,f)$, a contradiction.\\

If (iv) occurs, we have $F(z)\equiv jG(z)$, that is
\begin{align}
\frac{b(z+nc)-a(z+nc)}{b(z)-a(z)}=j(\frac{f(z+nc)-a(z)}{f(z)-a(z)}),
\end{align}
where $j\neq0,1$ is a finite constant. Then it follows from above, (4.17) and (4.19) that $e^{np(z)}=je^{np(z)}$, therefore we have $j=1$, a contradiction.

If (v) occurs, we have
\begin{align}
N(r,\frac{1}{f(z)-a(z)})=S(r,f ).
\end{align}
Then by Lemma 2.8, (4.7), (4.14) and $b(z-c)\not\equiv a(z)$, we obtain $T(r,f)=S(r,f)$, a contradiction.\\

If (vi) occurs, we have
\begin{align}
 N(r,f(z))=N(r,F(z))+S(r,f)=S(r,f),
\end{align}
and hence  we can see from Theorem G and Remark 1  that $f(z)\equiv f(z+c)$, a contradiction.\\

Therefore, $F(z)$ is not a linear fraction transformation of $G(z)$.   If $b(z)$ is a small function with period $nc$, that is $b(z+(n-1)c)\equiv b(z-c)$, we can set
\begin{eqnarray*}
\begin{aligned}
D(z)&=(f(z)-b(z))(b(z+nc)-b(z+(n-1)c))\notag\\
&-(f(z+nc)-b(z+nc))(b(z)-b(z-c))\notag\\
&=(f(z)-b(z-c))(b(z+nc)-b(z+(n-1)c))\notag\\
&-(f(z+nc)-b(z+(n-1)c))(b(z)-b(z-c))
\end{aligned}
\end{eqnarray*}
If $D(z)\equiv0$, then we have $f(z+nc)-b(z-c)\equiv -(f(z)-b(z-c))$. And thus we know that  $f(z)$ and $f(z+nc)$ share $a(z), b(z-c)$ and $\infty$ CM. We suppose
\begin{align}
F_{1}(z)=\frac{f(z)-a(z)}{b(z-c)-a(z)}, G_{1}(z)=\frac{f(z+nc)-a(z)}{b(z-c)-a(z)}.
\end{align}
Then we know that $F_{1}(z)$ and $G_{1}(z)$ share $0,1,\infty$ CM almost and $G_{1}(z)=-F_{1}(z)$. So by Lemma 2.10, we will obtain either $N(r,f(z))=N(r,F_{1})+S(r,f)=S(r,f)$, but in this case, according to Theorem G and Remark 1, we  can deduce a contradiction. Or $F_{1}(z)=G_{1}(z)$, that is $f(z)\equiv f(z+nc)$. Therefore, we obtain $f(z)\equiv b(z-c)$, that is $T(r,f(z))=S(r,f)$, a contradiction.

Hence $D(z)\not\equiv0$, and by (4.7)-(4.8), (4.14) and Lemma 2.1, we have
\begin{align}
2T(r,f(z))&=m(r,\frac{1}{f(z)-b(z)})+m(r,\frac{1}{f(z)-b(z-c)})+S(r,f)\notag\\
&=m(r,\frac{1}{f(z)-b(z)}+\frac{1}{f(z)-b(z-c)})+S(r,f)\notag\\
&\leq m(r,\frac{D(z)}{f(z)-b(z)}+\frac{D(z)}{f(z)-b(z-c)})+m(r,\frac{1}{D(z)})+S(r,f)\notag\\
&\leq m(r,D)+N(r,D)\leq m(r,f(z))+N(r,f(z))+S(r,f)\notag\\
&= T(r,f)+S(r,f),
\end{align}
which implies $T(r,f)=S(r,f)$, a contradiction.

By (4.16) we have
\begin{align}
\frac{\Delta_{c}b(z)}{1-e^{p(z)}}+b(z)=a(z).
\end{align}
Combining (4.18) and the fact that $a(z)$ is a small function with period $c$, we can get
\begin{align}
\frac{\Delta_{c}b(z+c)}{1-e^{p(z)}}+b(z+c)=a(z).
\end{align}
According to (4.27) and (4.28), we obtain
\begin{align}
e^{p(z)}=\frac{b_{2c}(z)-b_{c}(z)}{\Delta_{c}b(z)}.
\end{align}
So if $\rho(b(z))<\rho(e^{p(z)})$, we can follows from (4.28) and Lemma 2.11 that
\begin{align}
\rho(e^{p(z)})=\rho(\frac{b_{2c}(z)-b_{c}(z)}{\Delta_{c}^{2}b(z)})\leq \rho(b(z))<\rho(e^{p(z)}),
\end{align}
which is a contradiction.

 If $\rho(b(z))<1$, we claim that $p(z)\equiv B$ is a non-zero constant. Otherwise, the order of right hand side of (4.28) is $0$, but the left hand side is $1$, which is impossible. Therefore, by (4.1) we know that $f(z+c)-a(z)=B(f(z)-a(z))$, and then by Lemma 2.10 we will get $N(r,f)=S(r,f)$, so by Theorem G and Remark 1 we can obtain $f(z)\equiv f(z+c)$, a contradiction.

This completes Theorem 2.

\

{\bf Acknowledgements} The author would like to thank to referee  for his helpful comments.

\end{document}